\documentclass[10pt,a4paper,British]{article}
\usepackage{amsmath}
\usepackage{amssymb}
\usepackage{amsthm}
\usepackage{graphicx}
\newtheorem{theorem}{Theorem}
\newtheorem{corollary}{Corollary}
\newtheorem{lemma}{Lemma}

\DeclareSymbolFont{AMSb}{U}{msb}{m}{n}
\DeclareMathSymbol{\bbbn}{\mathalpha}{AMSb}{"4E}
\DeclareMathSymbol{\bbbz}{\mathalpha}{AMSb}{"5A}
\DeclareMathSymbol{\bbbr}{\mathalpha}{AMSb}{"52}
\DeclareMathSymbol{\bbbq}{\mathalpha}{AMSb}{"51}
\DeclareMathSymbol{\bbbc}{\mathalpha}{AMSb}{"43}

\DeclareMathOperator{\unb}{unb}

\author{ Paul Verschueren$^{\ast}$\thanks{$^{\ast}$Corresponding author: paul@verschueren.org.uk}
and Ben  D. Mestel$^{\dag}$\thanks{$^{\dag}$Ben.Mestel@open.ac.uk}
\\Department of Mathematics and Statistics\\ The Open University\\ Milton Keynes\\ MK7 6AA, UK
}

\title{Fixed points of Composition Sum Operators}

\makeatother

\begin{document}


\maketitle
\begin{abstract}
In the renormalisation analysis of critical phenomena in quasi-periodic systems, a fundamental
role is often played by fixed points of functional recurrences of the form
\begin{equation*}
f_{n}(x) = \sum_{i=1}^\ell a_i(x) f_{n_i} (\alpha_i(x)) \,,
\end{equation*}
where the $\alpha_i$,  $a_i$ are known functions and the $n_i$ are given and satisfy $n-2 \le n_i \le n-1 $.
We develop a general theory of fixed points of ``Composition Sum Operators'' derived from such recurrences, and
apply it to test for fixed points in key classes of complex analytic functions with singularities.
Finally we demonstrate the construction of the full space of fixed points of one important class, for the much studied operator
\begin{equation*}
Mf(x) = f(-\omega x) + f(\omega^2 x + \omega)\,, \quad  \omega = (\sqrt{5}-1)/2\,.
\end{equation*}
The construction reveals previously unknown solutions.
\vskip 2\baselineskip


\end{abstract}
\newpage{}


\section{Introduction\label{sec:intro}}


A composition sum operator $T$ is a linear operator acting on a vector
space $F$ of functions such that, for each function $f\in F$, the
image vector $Tf$, evaluated at $x$, is given by 
\begin{equation}
Tf(x)=\sum_{i=1}^{\ell}a_{i}(x)f(\alpha_{i}(x))\,,\label{eq:CSOs}
\end{equation}
where $\ell\ge1$, $\alpha_{1},\alpha_{2},\dots,\alpha_{\ell}$ are
affine contractions, and $a_{1}$, $a_{2}$, \dots, $a_{\ell}$ are
a fixed sequence of coefficients associated with $\alpha_{1},\alpha_{2},\dots,\alpha_{\ell}$.
A full definition is given in Section~\ref{sec:CSOs}. Although written
here as functions $a_{i}(x)$, in many applications the $a_{i}$ are
constants.

Composition sum operators have been studied extensively by Kuczma
and his co-workers. In particular, in the seminal monograph~\cite{Kuczma-et-al1990},
CSOs are discussed in detail in Chapter~6 ``Higher order equations
and linear systems'', principally in the real domain and in the ``cyclic
equation'' case, in which the $\alpha_{i}$ are iterates of a single
function. Kuczma \emph{et al} give some important existence and uniqueness
theorems in this context and we refer the reader to~\cite{Kuczma-et-al1990}
and to the references contained therein.

CSOs arise in several contexts, including the application of renormalisation
techniques to quasi-periodic non-linear dynamical systems and a toy-model
of magnetic flux growth in kinematic dynamo theory~\cite{Gilbert2002,Gilbert2005}.

Quasi-periodic systems are an important class of non-linear dynamical
systems which find application in many areas of the physical sciences.
In the simplest case, the dynamics are governed by an irrational number
$\omega\not\in\bbbq$, often called the rotation number or winding
number in the literature. It can often be identified as the ratio
of two incommensurate frequencies in the underlying system. Studies
of quasi-periodic systems often focus on the time-correlations between
system variables. These correlations, and, indeed other properties
of quasi-periodic systems, typically depend on the number-theoretic
properties of $\omega$, and, in particular, on the continued-fraction
expansion of $\omega$ and the associated rational convergents $p_{n}/q_{n}\to\omega$.
Examples of quasi-periodic systems include strange non-chaotic attractors,
the Harper equation and its generalisations, and other quantum mechanical
models depending on an underlying irrational rotation of the circle.

The correlation structure of quasi-periodic systems may be understood
by renormalisation analysis, leading to dynamical functional equations
which relate correlations at time $t$ to those at time $t+q_{n}$,
the dynamical properties of which depend on the dynamical behaviour
of the Gauss map applied to $\omega$ or, equivalently, on the action
of the shift map on the entries in the continued-fraction expansion
of $\omega$. In such studies the case of the golden-mean rotation
number, for which $\omega=(\sqrt{5}-1)/2$, often plays a pivotal
role. This is perhaps not surprising given the simplicity of its continued
fraction $[1,1,1,\dots]$, with all entries equal to $1$. For the
golden-mean, renormalisation analysis frequently leads to fixed-point
functional equations.

For example, renormalisation of correlations for the golden-mean Harper
equation leads to the so-called strong-coupling fixed point, satisfying
the functional equation 
\begin{equation}
f(z)=f(-\omega x)\, f(\omega^{2}x+\omega)\label{eq:SCFP}
\end{equation}
where $f$ is an analytic function with a pole of order $2$ at $z=1$
and $\omega=(\sqrt{5}-1)/2$. The construction of the strong-coupling
fixed point involves first studying fixed points of the composition
sum operator $M$ given by 
\begin{equation}
Mf(z)=f(-\omega z)+f(\omega^{2}z+\omega)\label{eq:M}
\end{equation}
where $f$ is analytic with a logarithmic singularity at $1$. (Here,
of course, $\ell=2$ and $a_{1}=a_{2}=1$ in equation~\eqref{eq:CSOs}.)

Note that equation~\eqref{eq:SCFP} is the fixed point case $f_{n}=f_{n-1}=f_{n-2}=f$
of the second-order multiplicative functional recurrence 
\begin{equation}
f_{n}(z)=f_{n-1}(-\omega x)\, f_{n-2}(\omega^{2}x+\omega)\,.\label{eq:multiplicative recurrence}
\end{equation}
Similarly, the associated linear recurrence 
\begin{equation}
f_{n}(z)=f_{n-1}(-\omega x)+f_{n-2}(\omega^{2}x+\omega)\label{eq:additive recurrence}
\end{equation}
leads in turn to the operator $M$. The functional recurrences~\eqref{eq:multiplicative recurrence}
and~\eqref{eq:additive recurrence} arise in several contexts involving
the golden-mean rotation number, in particular in connection with
the Ketoja-Satija orchid flower for the generalised Harper equation~\cite{Ketoja1995,Mestel2004a},
and, with piecewise constant functions $f_{n}$, in the analysis of
quantum two-level systems~\cite{Feudel-Pikovsky-Zaks1995}, barrier
billiards~\cite{Chapman2003} and strange non-chaotic attractors~\cite{Feudel-Pikovsky-Politi1996,Mestel2002}.
We refer to the comprehensive book by Feudel \emph{et al}~\cite{Feudel-Kuznetsov-Pikovsky2006}
for a full discussion of the applications of renormalisation theory
in strange non-chaotic attractors and to~\cite{Osbaldestin-Mestel2003}
for an overview of applications of equations~\eqref{eq:multiplicative recurrence}
and~\eqref{eq:additive recurrence}.

For other quadratic irrationals with constant continued-fraction expansion,
say $\omega=[a,a,a,a,\dots]$ where the integer $a\ge1$, we obtain
the multiplicative and additive fixed-point equations 
\begin{equation}
f(x)=\left(\prod_{i=0}^{a-1}f(-\omega x-i)\right)\; f(\omega^{2}x+a\omega)\,,\quad f(x)=\left(\sum_{i=0}^{a-1}f(-\omega x-i)\right)+f(\omega^{2}x+a\omega)\label{eq:multip_and_add_general_a}
\end{equation}
again with associated functional recurrences. See~\cite{Mestel2004b,Dalton-Mestel2003}
for applications in this case.

We now return to the golden-mean case. In~\cite{Mestel2000}, a rigorous
analysis established \emph{inter alia} the existence and properties
of a unique solution of~\eqref{eq:SCFP} under the constraints of
that physical situation, and provided an explicit expansion for the
strong-coupling fixed point. Indeed, writing $\phi_{1}(z)=-\omega z$
and $\phi_{2}(z)=\omega^{2}z+\omega$, this fixed point of $M$ has
the form: 
\begin{eqnarray*}
f(z) & = & \lambda\left[\log\left(\frac{1-z}{1-\omega}\right)+\sum_{k=1}^{\infty}\;\sum_{\substack{i_{1},...,i_{k}\\
i_{1}=1
}
}\log\frac{1-\phi_{i_{1}}\circ...\circ\phi_{i_{k}}(z)}{1-\phi_{i_{1}}\circ...\circ\phi_{i_{k}}(\omega)}\right]
\end{eqnarray*}
where $i_{j}\in\{1,2\}$, $\lambda\in\mathbb{C},$ $\log$ is the
principal branch of the logarithm, and $\circ$ signifies functional
composition.

The proof in~\cite{Mestel2000} is non-trivial and also depends on
properties of the golden mean so that its generalisation is not obvious.
The theory presented in this paper provides a general framework for
the construction of fixed-points of composition sum operators, which
not only illuminates the results of~\cite{Mestel2000}, but enables
the construction of fixed-points of other renormalisation operators
in a simplified and unified manner.

The organisation of this paper is as follows. We first describe in
Section~2 a general theory for the construction of fixed points of
linear operators on vector spaces. In Section~3 we introduce formally
the Composition Sum Operators (CSOs), describe their properties, and
define the function spaces of analytic functions on which we shall
work. We introduce the idea of a seed function, which we will use
extensively in the construction of fixed points of these operators.
In Section~4, we apply the theory in Section~2 to construct fixed
points of CSOs in the constant coefficient affine case (which we simply
call affine CSOs). Finally, in Section~5, we show how the theory
can be applied to construct fixed-points arising in the renormalisation
theory of quasi-periodic systems. This final section uses a construction
method derived from the methods in~\cite{Mestel2000}. We conclude
with some directions for future research.


\section{Seeded fixed point theory on vector spaces \label{sec:seeded-fixed-point-thy}}

In this section we describe the formal abstract setting for our construction
of fixed-points of linear operators on vector spaces. Our goal is
to derive from a given linear operator $T$, a \emph{fixed point operator}
$\widehat{T}$ which maps its domain (called the \emph{seed space}
of $T$) to the \emph{fixed point space} of $T$ (denoted $FP(T)$).

Although we have in mind applications to composition operators on
spaces of complex analytic functions with singularities, the theory
is quite general and may used in cases in which linear operators act
on a vector space that may be decomposed into a direct sum of subspaces,
one with a well defined Banach space structure and no non-zero fixed
points, and the other consisting of unbounded or singular vectors,
but which are prototype fixed points or ``seeds''. Although, set
in this general context, the theory is straightforward, its power
lies in its application to construct fixed points of renormalisation
operators and other operators, in which the fundamental structure
of the fixed points are evident, but the precise detail is not.

Let $F$ be a vector space and let $G\subset F$ be a proper non-zero
subspace of $F$. In many cases $G$ is equipped with norm $\|\cdot\|$
which endows $G$ with a Banach space structure, but this is not necessary
for the general theory. Let $T:F\to F$ be a linear operator and let
$\overline{T}$ denote the operator $I-T$ on $F$. We note that $f\in F$
is a fixed point of $T$ if, and only if, it is in the kernel of $\overline{T}$.

We assume that $T$ satisfies the following two properties. 
\begin{enumerate}
\item P1 $\overline{T}(F)\subset G$, so that $\overline{T}$ maps the whole
of $F$ into the subspace $G$. 
\item P2 The restricted operator $\overline{T}_{|G}$ is invertible on $G$
so that $T^{+}=\overline{T}_{|G}^{-1}$ exists and maps $G$ to $G$. 
\end{enumerate}
We note three points. First, when $G$ has a Banach space structure
with norm $\|\cdot\|$, then the second condition is satisfied when
$T_{|G}$ is a contraction with $\|T\|<1$ (but this is not a necessary
condition). Second, although these two conditions are very general,
one can often think of $G$ as the well-behaved non-singular part
of $F$ and $F\backslash G$ as being the singular or unbounded part
of $F$, which provides seeds for the construction of non-zero fixed
points of $T$. Third, writing $F$ explicitly as the direct sum $F=G\oplus S$,
we will see below that the vector space $S$ is a subspace of the
seed space which is mapped one to one to the fixed points of $T$
by the fixed point operator.

The following is the principal result for the construction of fixed
points of the operator $T$.

\begin{theorem}\label{thm:fixed_point_thm} Let $F$ be a vector space
and $T:F\to F$ be a linear operator satisfying the conditions P1
and P2 above. Then the linear operator $\widehat{T}:F\to F$ given
by 
\[
\widehat{T}=I-T^{+}\overline{T}
\]
maps $F$ to the subspace $FP(T)$ of fixed points of $T$ and $\widehat{T}$
induces a vector space isomorphism from the factor space $F/G$ to
$FP(T)$. In the case when $F=G\oplus S$, $\widehat{T}$ induces
an isomorphism from $S$ to $FP(T)$. \end{theorem}

The proof of this theorem is quite straightforward and belies the
utility and power of the theorem itself.

\begin{proof}

Let $f\in F$ and consider $\widehat{T}f$. Then $\overline{T}(\widehat{T}f))$
$=$ $\overline{T}f-\overline{T}(T^{+}\overline{T})f$ = $\overline{T}f-\overline{T}f=0$
so that $\widehat{T}f$ is a fixed point of $T$. Conversely, let
$f\in F$ be a fixed point of $T$. Then $\overline{T}f$ $=$ $0$
so that $\widehat{T}f=f$, and so $f\in\widehat{T}F$. It is straightforward
to show that $FP(T)$ is a linear subspace of $F$.

Now let us abuse notation slightly and also denote by $\widehat{T}$
the map $\widehat{T}:F/G\to FP(T)$ given by $\widehat{T}[f]=\widehat{T}f$,
for $[f]$ an element of $F/G$. It is straightforward to show that
$\widehat{T}$ is a linear map and we note that this map is well defined
because $G$ is in the kernel of $\widehat{T}$. It is immediate that
$\textrm{Im}~\widehat{T}$ $=FP(T)$ . Now let $[f]\in F/G$ be in
$\textrm{ker}~\widehat{T}$. Then $\widehat{T}f=0$, so that $f=T^{+}\overline{T}f$
whence $f\in G$ and so $[f]=[0]$. It follows that $\widehat{T}$
is a vector space isomorphism.

Finally, in the case $F=G\oplus S$, $S$ is isomorphic to $F/G$
via the natural inclusion, and so $\widehat{T}$ induces a vector
space isomorphism $S$ to $FP(T)$. \end{proof}

Note that $\widehat{T}$ is now our fixed point operator derived from
$T$, and its seed space (domain) is $\overline{T}^{-1}G$, ie $f$
is a seed if, and only if, $f\in\overline{T}^{-1}G$. This also means
$\widehat{T}f=f+g$ for $g=-T^{+}\overline{T}f\in G$.

There is also a straightforward but important extension which allows
us to extend the operator $\widehat{T}$.

\begin{corollary}\label{cor:FixedPointThmExtension} Suppose that condition
P2 holds, but not P1, so that we do not necessarily have $T(F)\subseteq G$.
Suppose, instead, that for some $f\in F$, $\overline{T}f\not\in G$,
but $\overline{T}(T^{k}f)\in G$ for some integer $k\ge1$. Writing
$f_{k}=\sum_{i=0}^{k-1}\overline{T}(T^{i}f)$, then $\overline{T}(f-f_{k})\in G$
and $\widehat{T}(f-f_{k})$ is a fixed point of $T$. Moreover, the
fixed point is independent of the choice of $k$. \end{corollary} \begin{proof}Since
$f_{k}=\sum_{0}^{k-1}\overline{T}(T^{i}f)=f-T^{k}f$, it is immediate
that $\overline{T}(f-f_{k})=\overline{T}(T^{k}f)\in G$, and so $f_{k}-f$
is a seed and we may now apply Theorem~\ref{thm:fixed_point_thm}.
The final statement follows from the observation that, if $\tilde{k}\ge k$,
then $\overline{T}T^{\tilde{k}}f$ $=$ $T^{\tilde{k}-k}\overline{T}T^{k}f\in G$,
since $T(G)\subseteq G$. It follows that $\widehat{T}(f-f_{\tilde{k}})-\widehat{T}(f-f_{k})$
$\widehat{T}(f_{k}-f_{\tilde{k}})$ $=0$, since $f_{k}-f_{\tilde{k}}\in G$.
\end{proof} We call a vector satisfying the hypotheses of Corollary\ref{cor:FixedPointThmExtension}
a \emph{generalised seed}.

As a simple application of Theorem~\ref{thm:fixed_point_thm}, we
consider the operator $T$ on real functions $c(x)$ defined on $[-1,1]$
\[
Tc(x)=c\left(\frac{x-1}{2}\right)-c\left(\frac{1-x}{2}\right)\,.
\]
This operator arises from the zero-shear base case of a the Stretch-Fold-Shear
toy model in kinematic dynamo theory, studied in detail by Gilbert~\cite{Gilbert2002,Gilbert2005}.

Now, for integer $n\ge1$, let $P_{2n-1}^{o}$ denote the real vector
space of odd polynomials of degree at most $2n-1$. Then, evidently,
$TP_{2n-1}^{o}\subseteq P_{2n-1}^{o}$, and, indeed, $T$ has an upper-triangular
matrix with respect to the standard basis $\{x,x^{3},\dots,x^{2n-1}\}$,
from which the spectrum of $T$ restricted to $P_{2n-1}^{o}$ is readily
obtained. Let us now consider the operator $T$ from the viewpoint
of Theorem~\ref{thm:fixed_point_thm}. Writing $P_{2n-1}^{o}=P_{2n-3}^{o}\oplus<x^{2n-1}>$
and $T_{2}=4^{n-1}T$, then it is straightforward to verify that the
hypotheses of Theorem~\ref{thm:fixed_point_thm} are satisfied with
$G=P_{2n-3}^{o}$ and $S=<x^{2n-1}>$, from which the spectrum and
eigenfunctions of $T$ on $P_{2n-1}^{o}$ may be calculated. In fact,
this is also the spectrum of $T$ acting on a more general space of
analytic functions on which $T$ is compact. See~\cite{Gilbert2002,Gilbert2005}
for details.


\section{Composition sum operators, analytic function spaces, and seed functions
\label{sec:CSOs}}

We can now apply the general theory of the previous section to help
us identify fixed points of the particular class of operators we call
\emph{Composition Sum Operators}. The operators $M$ and $T$ introduced
above are examples of this class.

Let $\alpha:D\longrightarrow D$ be a map of a complex domain into
itself, and let $F$ be a ring of complex-valued functions defined
on $D$. In practical applications, $D$ is frequently a disc and
$F$ a space of analytic functions on $D$, possibly with singularities.

For any $f\in F$ we denote by $\alpha^{*}f$ the function $f\alpha$,
the composition of $f$ with $\alpha$. Then $\alpha^{*}$ is an operator
on $F$, which we call a \emph{Composition Operator} on $F$.

A \emph{Composition Sum Operator (CSO)} on $F$ is an operator $T=\sum_{i=1}^{\ell}a_{i}\alpha_{i}^{*}$
on $F$ where $a_{i}\in F$, $a_{i}\not=0$, $\alpha_{i}^{*}$ is
a Composition Operator on $F$, and $\left(\sum a_{i}\alpha_{i}^{*}\right)f=\sum a_{i}.f\alpha_{i}$.
We call the positive integer $\ell$ the \emph{length} of the CSO,
and we assume that $\alpha_{i}\not=\alpha_{j}$, for $i\not=j$. When
it is clear from the context, we suppress the explicit range $i=1$,
$\dots$, $\ell$.

In the case when $F$ is Banach space with norm $\left\Vert \cdot\right\Vert $,
we have $\left\Vert Tf\right\Vert =\left\Vert \sum a_{i}.f\alpha_{i}\right\Vert \le\sum\left\Vert a_{i}\right\Vert \left\Vert f\right\Vert $
so $\left\Vert T\right\Vert \le\sum\left\Vert a_{i}\right\Vert <\infty$
so a CSO is also a bounded linear operator on these spaces. However,
the most interesting cases occur when $T$ is an operator on a space
of functions with singularities.

Although CSOs may be defined for general maps $\alpha_{i}$, in most
applications they are affine contractions and the coefficients $a_{i}$
are constant functions. We say that the CSO $T=\sum_{i=1}^{\ell}a_{i}\alpha_{i}^{*}$
is \emph{affine} if each $a_{i}$ is a constant, and each $\alpha_{i}$
is an affine contraction, ie., $\alpha_{i}(z)=s_{i}(z-z_{i})+z_{i}$,
where $z_{i}\in D$ and $|s_{i}|<1$, $s_{i}\in\bbbc$.

Affine CSOs are the principal application of the seeded fixed point
theory described above. In the next section we develop the theory
of seed functions to construct fixed points of CSOs with singularities
of pole and logarithmic type. We concentrate on these as they currently
seem the most significant; however the techniques presented are readily
extended to other types of singularity such as removable and algebraic
singularities.

\subsection{Seed functions over spaces of bounded functions}

In this section we look at the conditions under which a complex function
$f$ with various types of singularities can be a seed for a CSO $T$
over a Banach space of analytic functions $G$. In particular this
requires that $\overline{T}f\in G$. This is a strong condition as
we shall see. In this context we refer to seeds as \emph{seed functions}.
We also recall the important result that that $\widehat{T}f=f+g$
for some $g\in G$.

We will show that, apart from a small set of CSOs which admit polynomials
as fixed points, there are no analytic non-zero fixed points of an
affine CSO in $G$. All other non-zero fixed points therefore have
singularities of some sort. For some CSOs with real coefficients the
singularities can be discontinuities (see~\cite{Mestel2002} for
an example), but in the context of CSOs defined on spaces of analytic
functions (with singularities) it is logarithmic and unbounded isolated
singularities which are of greatest interest, and we will analyse
these below.

To be precise, by unbounded isolated singularities we mean poles and
essential singularities, ie not the singularities of multivalued functions
such as $\log z$, or $z^{\alpha}$ for non-integral $\alpha$. And
by logarithmic singularities, we mean the unbounded singularities
of functions which can be written $\log g$ with $g$ analytic, ie
not composed log functions such as $\log\log z$ or $(\log z)^{\alpha}$.
We will refer to functions analytic apart from these singularities
as\label{PESL} (Pole/Essential/Simple Log) functions.

Consider a fixed affine CSO acting on functions on a domain $D\subseteq\bbbc$
and $G$ a Banach space of analytic functions on $D$. We consider
a complex valued function defined and analytic on an open dense subset
$D'$ of $D$. The set of points on which $f$ fails to be analytic
might include, for example, points on branch cuts other than branch
points. We therefore restrict our discussion to unbounded singularities
ie points at which $f$ fails to be bounded on any neighbourhood.

We say that $z_{0}\in D-D'$ is an \emph{unbounded singularity} if
$z_{0}$ has no neighbourhood $U\subseteq D$ on which $f_{|U\bigcap D'}$
is bounded, ie $\sup_{z\in U\bigcap D'}|f(z)|=\infty$ for any neighbourhood
$U$ of $z_{0}$. in $D$. The \emph{unbounded singularity set} of
$f$ in $D$, written $\unb_{D}(f)$ is the (possibly empty) set of
points in $D$ at which $f$ has an unbounded singularity.

Recall that $f$ is a seed function if $\overline{T}f\in G$. It follows
that, if $f$ is a seed function, then $\unb_{D}(f)$ $=$ $\unb_{D}(Tf)$,
because each function in $G$ has no unbounded singularities in $D$.

In principle, $\unb_{D}(f)$ may be quite large and composition operators
may act and interact on the set in intricate ways which are beyond
the scope of this paper. We will restrict our attention to \emph{simple\label{simple}}
actions which we define as follows: we will say that $\unb_{D}(f)$
is \emph{simple} under $T$ if (i) each $\alpha_{i}$ acts unstably
on $\unb_{D}(f)$, (ie for $z\in unb_{D}(f)$, $\alpha_{i}(z)\not\in unb_{D}(f)$
unless $\alpha_{i}(z)=z$) and (ii) if $\alpha_{i}(z)=z$ then $\alpha_{j}(z)\ne z$
for $j\ne i$. We also say $f$ itself is simple under $T$ if $\unb_{D}(f)$
is simple under $T$.

From these results the following lemma follows readily. \begin{lemma}
\label{lem:UnboundedPoints}Let $T=\sum_{i}a_{i}\alpha_{i}^{*}$ be
an affine CSO, and let $f$ be a seed over a Banach space $G$ of
analytic functions. If the unbounded set of $f$ is simple under $T$,
then $\unb_{D}(f)\subseteq\bigcup_{i}FP(\alpha_{i})$, and each unbounded
point of $f$ is a fixed point of precisely one $\alpha_{i}$. \end{lemma}

\begin{proof} If $z_{0}\in\unb(f)$ then $z_{0}\in\unb(Tf)$ so $\sum_{i}a_{i}.f\alpha_{i}$
is unbounded at $z_{0}$, so, for at least one $i$, $f\alpha_{i}$
is unbounded at $z_{0}$. Since the unbounded set of $f$ is simple,
this means $\alpha_{i}(z_{0})=z_{0}$ and the $i$ is unique. \end{proof}

As an example, consider the CSO $M$ above. Its fixed points are $\{0,1\}$.
So any simple seed or fixed point of $M$ is unbounded on at most
$\{0,1\}$.

\section{Fixed-point construction for affine CSOs}

Recall that an affine CSO is a CSO for which the functions $a_{1}$,
$a_{2}$, $\dots$, $a_{\ell}$ are non-zero constants and the maps
$\alpha_{1}$, $\alpha_{2}$, $\dots$, $\alpha_{\ell}$ are affine
contractions $\alpha_{i}(z)=z_{i}+s_{i}(z-z_{i})$ on the complex
plane, where the fixed points $z_{i}$ and and contraction rates $s_{i}$
are all complex constants, with $0\le|s_{i}|<1$. (Here, and in what
follows, $i$ ranges from $1,\dots,\ell$.) In many applications the
$z_{i}$ and the $s_{i}$ are real, but the theory may be just as
easily developed for complex $z_{i}$ and $s_{i}$. With affine CSOs
we are able to obtain a good theory for the construction of fixed
points, drawing on the work of the previous sections.

We define here an additional constraint which which will prove useful
in this section. Given a CSO $T=\sum a_{i}\alpha_{i}^{*}$ on a domain
$D$, we will say $\alpha_{i}$ is \emph{fixed point independent on
$D$} if $z_{i}\not\in\bigcup_{j\ne i}\overline{\alpha_{j}(D)}$.
This means that if $z_{i}$ is a singularity of a function $f$, $f\alpha_{j}$
has a singularity at $z_{i}$ if, and only if, $j=i$.

We shall work in a fixed disc in the complex plane. Let $D=D_{r}$,
the open disc of radius $r$ about $0$ in $\bbbc$, and let $G(D_{r})$
be the complex Banach space of functions $g$ analytic on $D_{r}$
with finite supremum norm $||g||_{\infty,r}=\sup\{|g(z)|\;:\; z\in D_{r}\}$.
Let $R>0$ be chosen so that for some $\delta>0$, $\overline{\alpha_{i}(D_{R+\delta})}\subseteq D_{R-\delta}$
for all $i=1,\dots,\ell$. Because the $\alpha_{i}$ are contractions,
this condition holds provided we take $R$ sufficiently large. We
write $D=D_{R}$, and $G=G(D_{R})$ and $||g||_{\infty}$ $=$ $||g||_{\infty,R}$,
for $g\in G(D_{R})$.

Let us consider the affine Composition Sum Operator $Tf(z)=\sum_{i=1}^{\ell}a_{i}f(\alpha_{i}(z))$.
It is straightforward to verify that $T$ is a linear operator on
the complex Banach Space $G$. Moreover, for $m\ge0$, we may differentiate
$m$ times the function $Tg$: 
\begin{equation}
(Tg)^{(m)}(z)=\sum_{i=1}^{\ell}a_{i}s_{i}^{m}g^{(m)}(\alpha_{i}(z))\,.
\end{equation}
We now define an induced operator on $G$, which we denote by $T^{(m)}$,
given by 
\begin{equation}
T^{(m)}\tilde{g}(z)=\sum_{i=1}^{\ell}a_{i}s_{i}^{m}\tilde{g}(\alpha_{i}(z))\,,
\end{equation}
for $\tilde{g}\in G$. We have that 
\begin{equation}
||T^{(m)}\tilde{g}||_{\infty}\le\sum_{i=1}^{\ell}|a_{i}||s_{i}|^{m}||\tilde{g}||_{\infty}\,,
\end{equation}
so that the operator norm $||T^{(m)}||\le\sum_{i=1}^{\ell}|a_{i}||s_{i}|^{m}$.
An immediate consequence is that there exists $m\ge0$ such that $||T^{(m)}||<1$,
a contraction. Finally, using Cauchy estimates, we see that if $g\in G$,
then $g^{(m)}\alpha_{i}$ is also in $G$ and, moreover, $||g^{(m)}\alpha_{i}||_{\infty}\le K||g||_{\infty}$,
where $K=m!$ $\delta^{-(m+1)}$, for $i=1,\dots,\ell$.

From these results, we may readily show that all non-trivial fixed
points of $T$ in $G$ are polynomials. The proof is rather elegant.
Indeed, suppose $g\in G$ is a fixed point of $T$, with $g\not=0$.
Then, for some $m\ge0$, $0$ $\le$ $||g^{(m)}||_{\infty}$ $=$
$||T^{(m)}g^{(m)}||_{\infty}$ $\le$ $||T^{(m)}||\,||g^{(m)}||_{\infty}$
$<$ $||g^{(m)}||_{\infty}$, a contradiction. It follows that $g$
is zero or a polynomial of degree at most $m-1$.

Whether or not $T$ has a polynomial fixed point depends on the precise
values of the $a_{i}$ and $\alpha_{i}$. Indeed, for a polynomial
$p(x)=p_{0}+p_{1}x+\dots+p_{m}x^{m}$, $p_{m}\not=0$, it is clear
that $Tp$ is a polynomial of degree at most $m$. Inspecting the
coefficient of $x^{m}$ in $Tp(x)=p(x)$, we have 
\begin{equation}
a_{1}s_{1}^{m}+a_{2}s_{2}^{m}+\dots+a_{\ell}s_{\ell}^{m}=1\label{eq:degree-m-relation}
\end{equation}
which is clearly a necessary condition for a polynomial fixed point
of degree $m$. Conversely, suppose that~\eqref{eq:degree-m-relation}
holds. Then if $p(x)$ is of degree $m$, $\overline{T}p$ is of degree
at most $m-1$, and so $\overline{T}$ is degenerate and has non-trivial
kernel. If $q$ is in the kernel, then $Tq=q$. Hence $T$ has a non-trivial
space of polynomial fixed points if, and only if, (\ref{eq:degree-m-relation})
holds for one or more $m\ge0$. Note that there is some $N>0$ such
that the condition does not hold for any $m\ge N$, and so the space
of polynomial fixed points of $T$ is of bounded maximum degree.

We now assume that there are no polynomial fixed points, ie that 
\begin{equation}
a_{1}s_{1}^{j}+a_{2}s_{2}^{j}+\dots+a_{\ell}s_{\ell}^{j}\not=1\,,\quad\textrm{for all \ensuremath{j\ge0}.}\label{eq:no-poly-soln}
\end{equation}
It is now evident that the only non-zero fixed points are necessarily
singular on $D$. In what follows we restrict ourselves to unbounded
isolated and logarithmic singularities.

Let us consider first simple seeds $f$ with unbounded isolated singularities.
Since $f$ is simple, every point of $\unb_{D}{f}$ is a fixed point
of a unique $\alpha_{i}$. Without loss of generality, we let $i=1$
and we suppose that $f$ has an isolated singularity at $z_{1}$,
so that $f_{|D-\{z_{1}\}}$ is analytic on some neighbourhood of $z_{1}$.
Let $C_{\epsilon}$ be a circle of radius $\epsilon>0$ about $z_{1}$.
Then if $\alpha_{i}$ is fixed point independent of the other $\alpha_{j}$
on $D$, and for $\epsilon$ sufficiently small, $f$ is analytic
inside and on $\alpha_{i}(C_{\epsilon})$ for $i=2,\dots,\ell$. Using
the fact that $f-Tf$ is analytic and integrating along $C_{\epsilon}$,
we have, for integer $k\ge0$, 
\begin{align}
0 & =\int_{C_{\epsilon}}(z-z_{1})^{k}(f(z)-Tf(z))dz=\int_{C_{\epsilon}}(z-z_{1})^{k}(f(z)-a_{1}f(\alpha_{1}(z)))dz\\
 & =\left(\int_{C_{\epsilon}}(z-z_{1})^{k}f(z)dz-\int_{C_{\epsilon}}(z-z_{1})^{k}a_{1}f(\alpha_{1}(z))dz\right)\\
 & =\left(\int_{C_{\epsilon}}(z-z_{1})^{k}f(z)dz-\int_{\alpha_{1}^{-1}C_{\epsilon}}s_{1}^{-(k+1)}(w-z_{1})^{k}a_{1}f(w)dw\right)\\
 & =2\pi if_{-(k+1)}\left(1-s_{1}^{-(k+1)}a_{1}\right)\,.
\end{align}
In this calculation we have used Cauchy's integral theorem, together
with a change of variable $w=\alpha_{1}(z)$. We have denoted by $f_{-(k+1)}$
the $(k+1)-$th coefficient in the Laurent expansion of $f$ about
$z_{1}$.

We conclude for $k\ge0$ that either $f_{-(k+1)}=0$ or $a_{1}=s_{1}^{(k+1)}$.
If the latter condition holds, then we may have $f_{-(k+1)}\not=0$
and $f(z)=(z-z_{1})^{-(k+1)}$ is a seed function, from which a fixed
point of $T$ may be constructed, provided equation~\eqref{eq:no-poly-soln}
holds. The construction is omitted here as it is similar to that given
below for the logarithmic case.

The result also shows that essential singularities do not lead to
fixed points of affine CSOs. For $\left(1-s_{1}^{-(k+1)}a_{1}\right)=0$
cannot hold for more than one $k\ge0$, ruling out a non-finite principal
part of $f$ at $z_{1}$.

We now give the construction of a fixed point of $T$ in the case
when the seed function $f$ has a single logarithmic singularity of
the form $f(z)=\log(z-z_{i})+g(z)$, where $g\in G$$ $ and $z_{i}\in D$,
and where $i$ is one of $1$, $2$, $\dots$, $\ell$. Again, without
loss of generality, we take $i=1$.

Our first observation is that we may take $f(z)=\log(z-z_{1})$, since
if $\tilde{f}(z)=\log(z-z_{1})+g(z)$, where $g\in G$, then $\widehat{T}f=\widehat{T}\tilde{f}$.
(Any convenient branch of the logarithm function may be taken, although,
to be specific, we choose the principal branch.) Again we let $\alpha_{1}$
be fixed point independent on $D$ so that, for $j\not=1$, $z_{1}\not\in\overline{\alpha_{j}(D)}$.
Therefore $f(z)-Tf(z)=\log(z-z_{1})-a_{1}\log(\alpha_{1}(z)-z_{1})+g(z)$,
where $g\in G$, and, since $\log(z-z_{1})-a_{1}\log(\alpha_{1}(z)-z_{1})$
$=\log(z-z_{1})-a_{1}\log(z_{1}+s_{1}(z-z_{1})-z_{1})$ $=(1-a_{1})\log(z-z_{1})-a_{1}\log s_{1}$,
it follows that $f-Tf\in G$ if, and only if, $a_{1}=1$.

Let us now assume that $a_{1}=1$ and $f(z)=\log(z-z_{1})$. For convenience
we consider separately the cases when $T$ is a contraction on $G$
and when $T$ is not a contraction on $G$.

The first case is easily handled directly by appealing to Theorem~\ref{thm:fixed_point_thm}.
Let $F=\langle\log(z-z_{1})\rangle\oplus G$. Then $T:F\to F$ satisfies
the hypothesis of Theorem~\ref{thm:fixed_point_thm}, from which
we conclude immediately the construction of a one-dimensional subspace
of fixed-points of $F$ in $F$ with logarithmic singularity $\log(z-z_{1})$.

The second case may be handled by differentiating the operator $T$,
say $m$ times, until it is a contraction, appealing to Theorem~\ref{thm:fixed_point_thm}
for the induced operator $T^{(m)}$, and then integrating up to obtain
a fixed point of $T$. Specifically, let $m\ge1$ be such that $\sum_{i=1}^{\ell}|a_{i}||s_{i}|^{m}<K<1$
and let $f_{m}(z)=(m-1)!(-1)^{m-1}(z-z_{1})^{-m}$. Then 
\begin{equation}
T^{(m)}f_{m}(z)-f_{m}(z)=\sum_{i=1}^{\ell}a_{i}s_{i}^{m}f_{m}(\alpha_{i}(z))-f_{m}(z)=\sum_{i=2}^{\ell}a_{i}s_{i}^{m}f_{m}(\alpha_{i}(z))\,,
\end{equation}
as may readily be ascertained by direct calculation. The right-hand
side is in $G$ (since $z_{1}\not\in\overline{\alpha_{i}(D)}$) for
$i=2,\dots,\ell$, so that $f_{m}$ is a seed function for $T^{(m)}$.
Moreover, $I-T^{(m)}$ is invertible in $G$ because $||T^{(m)}||\le K<1$.
We may therefore apply Theorem~\ref{thm:fixed_point_thm} with $F$
$=$ $<f_{m}>$ $\oplus$ $G$ to obtain a one-dimensional subspace
of fixed points $<\widehat{f_{m}}>$ of $T^{(m)}$ in $F$.

To obtain a fixed point of $T$, we integrate $m$ times, although
we must then handle a polynomial of degree at most $m-1$ that arises
from the constants of integration. Specifically, let us define the
integration operator $I:G\to G$ by the integral on the line segment
$[0,z]$ for $z\in D$: 
\begin{equation}
I(g)(z)=\int_{0}^{z}g(w)dw\,.
\end{equation}
Denoting the $m-$th iterate of $I$ by $I^{m}$, and noting that
$\widehat{f_{m}}-f_{m}\in G$, we may define the function $f+I^{m}(\widehat{f_{m}}-f_{m})$
which we denote $\hat{f}$. The function $\hat{f}$ is not necessarily
a fixed point of $T$. However, differentiating $T\hat{f}-\hat{f}$
$m$ times, we obtain 
\begin{align}
\left(T\hat{f}-\hat{f}\right)^{(m)} & =\left(Tf-f+TI^{m}(\widehat{f_{m}}-f_{m})-I^{m}(\widehat{f_{m}}-f_{m})\right)^{(m)}\\
 & =(T^{(m)}f_{m}-f_{m})+T^{(m)}(\widehat{f_{m}}-f_{m})-(\widehat{f_{m}}-f_{m})\\
 & =0\,,
\end{align}
since $T^{(m)}\widehat{f_{m}}=\widehat{f_{m}}$. It follows that $T\hat{f}-\hat{f}=q_{m}$,
where $q_{m}$ is a polynomial of degree at most $m-1$. Now let $p_{m}=(I-T)^{-1}q_{m}$,
a polynomial of degree at most $m-1$, the inverse existing because
of~\eqref{eq:no-poly-soln}. Then we have immediately that $T(\widehat{f}+p_{m})-(\widehat{f}+p_{m})=0$,
so that $\widehat{f}+p_{m}$ is a fixed point of $T$.

We have therefore proved the following theorem: \begin{theorem} Let $T$
be an affine Composition Sum Operator given by 
\[
T(f)(z)=\sum_{i=1}^{\ell}a_{i}f(\alpha_{i}(z))\,,
\]
where $\ell\ge2$ is an integer, and for $i=1,\dots,\ell$, $a_{i}$
$\in\bbbc$, and $\alpha_{i}(z)=s_{i}(z-z_{i})+z_{i}$ are affine
contractions. Let $R>0$ be such that there exists $\delta>0$ with
$\overline{\alpha_{i}(D_{R+\delta})}\subseteq D_{R-\delta}$ for $i=1,\dots,\ell$

Then 
\begin{enumerate}
\item $T$ has a fixed point which is a non-zero polynomial if and only
if $a_{1}s_{1}^{m}+a_{2}s_{2}^{m}+\dots+a_{\ell}s_{\ell}^{m}=1$ for
some integer $m\ge0$. If there are polynomial fixed points, there
is also a maximum integer $m$ satisfying the constraint, and all
the fixed points are then of degree at most $m$. 
\item If there are no polynomial fixed points%
\footnote{This condition guarantees the invertibility of $I-T$, and hence the
existence of PESL fixed points. However if polynomial fixed points
do exist, the possibility of PESL fixed points is not ruled out, and
if they do exist they will satisfy the conditions given above for
$a_{i}$. %
}, but for some $1\le i\le l$, $\alpha_{i}$ is fixed point independent
on $D_{R}$ (ie $z_{i}\not\in\bigcup_{j\ne i}\overline{\alpha_{j}(D_{R})}$),
then:

\begin{enumerate}
\item If $a_{i}=s_{i}^{k}$, for some $k\ge1$, then $T$ has a fixed point
$f$ of the form 
\[
f(z)=(z-z_{i})^{-k}+g(z)
\]
where $g\in G$ is analytic and bounded in $D_{R}$. 
\item If $a_{i}=1$, then $T$ has a fixed point $f$ of the form 
\[
f(z)=\log(z-z_{i})+g(z)
\]
where $g\in G$ is analytic and bounded in $D_{R}$. 
\end{enumerate}
\end{enumerate}
Moreover if every $\alpha_{i}$ of $T$ is fixed point independent
on $D_{R}$, then every simple fixed point of $T$ which is of PESL%
\footnote{Having only Pole/Essential/Simple Log singularities - see \ref{PESL}%
} type is necessarily a linear combination of fixed points satisfying
the conditions above.

\end{theorem}

\section{\label{sub:CaseStudy}Applications to problems arising from renormalisation
theory}

We now consider further the operator $M$ given in~\eqref{eq:M}
above and we discuss the construction of fixed points of $M$. The
approach we adopt differs from that in the previous section in that
we work directly from the operator $M$, modifying it by subtracting
a constant CSO to obtain a contraction. This is more in the spirit
of the work in \cite{Mestel2000,Dalton-Mestel2003}. We start by developing
a general theory of CSOs acting on $\ell_{1}$ spaces of analytic
functions, a theory which is complementary to that developed in Sections~3
and~4.

\subsection{CSOs acting on $\ell_{1}$ spaces of analytic functions}

For $R>0$, let $G_{R}$ denote the complex Banach space of analytic
functions on the open disc $D_{R}=\{z:\,|z|<R\}$ with finite $\ell_{1}-$norm
$\left\Vert \sum_{n=0}^{\infty}c_{n}z^{n}\right\Vert _{R}=\sum_{n=0}^{\infty}|c_{n}|R^{n}$.
For $n\ge0$, we denote by $Z_{n}$ the basis function $Z_{n}:\, z\longmapsto z^{n}$,
which has norm $R^{n}$. The set $\{Z_{n}\;:\; n=0,1,2,\dots\}$ forms
a basis for $G_{R}$. We note the following standard lemma, which
we include for completeness.

\begin{lemma} \label{R:TonZn} Let $T$ be a bounded linear operator
on the Banach space $G_{R}$ of analytic functions and let $K>0$.
Then the induced operator norm $\bigl\Vert T\bigr\Vert_{R}\le K$
if, and only if, $\bigl\Vert TZ_{k}\bigr\Vert_{R}\le K\left\Vert Z_{k}\right\Vert $
for all $k\ge0$. The result also holds when the inequality is replaced
with a strict inequality. \end{lemma}

It follows that $T$ is a contraction on $G_{R}$ with contraction
rate $K<1$ if, and only if, it contracts each basis function $Z_{k}$
with contraction rate $K$.

\begin{proof} First suppose $\bigl\Vert T\bigr\Vert_{R}\le K$. Since
$Z_{k}\in G_{R}$, $\bigl\Vert TZ_{k}\bigr\Vert_{R}$ $\le$ $\bigl\Vert T\bigr\Vert_{R}\bigl\Vert Z_{k}\bigr\Vert_{R}$
$\le$ $K.R^{k}$, as required. We now prove the converse. Let $f=\sum_{r=0}^{\infty}a_{r}Z_{r}\in G_{R}$.
Since $T$ is bounded, hence continuous, $Tf=T\lim_{n\to\infty}\sum_{r=0}^{n}a_{r}Z_{r}=\lim_{n\to\infty}T\sum_{r=0}^{n}a_{r}Z_{r}=\lim_{n\to\infty}\sum_{r=0}^{n}a_{r}TZ_{r}$
and so it follows that $\bigl\Vert Tf\bigr\Vert_{R}$ $=$ $\lim_{n\to\infty}\bigl\Vert\sum_{r=0}^{n}a_{r}TZ_{r}\bigr\Vert_{R}$
$\le$ $\lim_{k\to\infty}\sum_{r=0}^{k}|a_{r}|KR^{r}$ $=$ $K\bigl\Vert f\bigr\Vert_{R}$
whence $\bigl\Vert T\bigr\Vert_{R}\le K$, as claimed. This completes
the proof. \end{proof}

One particular feature of affine CSOs is that they are contractions
on the basis functions $Z_{k}$ for $k$ sufficiently large, as is
shown by the following result.

\begin{lemma} Let $T=\sum_{i=1}^{\ell}a_{i}\alpha_{i}^{*}$ be a CSO
with $a_{i}$ constant (and non-zero), $\alpha_{i}(z)=s_{i}z+t_{i}$
where $|s_{i}|<1$. Let $s=\max_{i}\{|s_{i}|\}$, and let $\mu\in\bbbr$
satisfy $s<\mu\le1$.

Then there exists $R_{0}\ge0$, and integer $N>1$ such that 
\begin{enumerate}
\item $\bigl\Vert TZ_{k}\bigr\Vert_{R}<\mu^{k}R^{k}$ for all $R>R_{0}$
and all $n\ge N$. 
\item For $0\le n<N$, $\left\Vert TZ_{n}\right\Vert _{R}<\mu R^{n}$ for
all $R>R_{0}$, whenever $\left|\sum_{i=1}^{\ell}a_{i}s_{i}^{k}\right|<\mu$. 
\end{enumerate}
\end{lemma}

\begin{proof} We have $(TZ_{n})(z)$ $=$ $\sum_{i}a_{i}(s_{i}z+t_{i})^{n}$
$=$ $\sum_{r=0}^{n}z^{r}\sum_{i}a_{i}\binom{n}{r}s_{i}^{r}t_{i}^{n-r}$,
hence it follows that $\left\Vert TZ_{n}\right\Vert _{R}$ $\le$
$\sum_{r=0}^{n}\left|\sum_{i}a_{i}\binom{n}{r}s_{i}^{r}t_{i}^{n-r}\right|.R^{r}$
$=$ $R^{n}\sum_{r=0}^{n}\binom{n}{r}\left|\sum_{i}a_{i}s_{i}^{r}\left(\frac{t_{i}}{R}\right)^{n-r}\right|$.
Therefore $\left\Vert TZ_{n}\right\Vert _{R}\le R^{n}\sum_{i}\left|a_{i}\right|.(\left|s_{i}\right|+\left|\frac{t_{i}}{R}\right|)^{n}$
from which we see that if $\left|s_{i}\right|+\left|\frac{t_{i}}{R}\right|<\mu\le1$
for all $i$, then we can find $N\ge0$ so that $\left\Vert TZ_{n}\right\Vert <\mu^{n}R^{n}$
for $n\ge N$. The condition on $R$ equates to $R_{0}=\max_{i}\{\frac{|t_{i}|}{\mu-|s_{i}|}\}$.

We now consider $n<N$ for $N>1$. From above $\left\Vert TZ_{n}\right\Vert _{R}\le R^{n}\left(\left|\sum_{i}a_{i}s_{i}^{n}\right|+\sum_{r=0}^{n-1}\binom{n}{r}\left|\sum_{i}a_{i}s_{i}^{r}\left(\frac{t_{i}}{R}\right)^{n-r}\right|\right)$.

Since $n$ is now bounded, for any $\epsilon>0$, we can choose $R$
large enough to give $\left\Vert TZ_{n}\right\Vert _{R}\le R^{n}\left(\left|\sum_{i}a_{i}s_{i}^{n}\right|+\epsilon\right)$.

Hence if $\left|\sum_{i}a_{i}s_{i}^{n}\right|<\mu\le1$ for $0\le n<N$,
we will have $\left\Vert TZ_{n}\right\Vert _{R}<\mu R^{n}$ for large
enough $R$. \end{proof}

The following corollaries are immediate. \begin{corollary}\label{cor:TcontractionforlargeR}
If $T$ is an affine CSO with $|s_{i}|<1$ and, for each $n\ge0$,
$\left|\sum_{i}a_{i}s_{i}^{n}\right|<1$, then $T$ is a contraction
on $G_{R}$ for large enough $R$. \end{corollary}

\begin{corollary}\label{cor:TcontractionforlargeR-1} If $T$ is an affine
CSO with $|s_{i}|<1$ and, for each $n\ge1$, $\left|\sum_{i}a_{i}s_{i}^{n}\right|<1$,
then $T_{c}$ is a contraction on $G_{R}$ for large enough $R$,
where $c$ is a constant and $T_{c}f=Tf-c^{*}(Tf)=\sum_{i=1}^{\ell}a_{i}f\alpha_{i}-\sum_{i=1}^{\ell}a_{i}f\alpha_{i}(c)$.
\end{corollary} \begin{proof} Note that for any constants $c$ and $a$,
$ac^{*}$ is a degenerate affine CSO with $s=0$, so that for $n\ge1$
the sum $\left|\sum_{i}a_{i}s_{i}^{n}\right|$ is unchanged between
$T$ and $T_{c}.$ But the sum is precisely $0$ for $n=0$, and so
the previous corollary can be applied to $T_{c}$.\end{proof}

For a seed function $f$, it may happen that the function $g=Tf-f\in G_{R}$
only for $R$ in a restricted range. In these circumstances it may
not be possible to apply Corollaries~\ref{cor:TcontractionforlargeR}
and~\ref{cor:TcontractionforlargeR-1} directly. Instead we may have
to iterate $T$ several times so that the domain on which $g$ is
defined is extended to include $D_{R}$ for $R$ sufficiently large
for Corollaries~\ref{cor:TcontractionforlargeR} and~\ref{cor:TcontractionforlargeR-1}
to apply. That it is possible to do this follows from the fact that
the $\alpha_{i}$ contract the whole of $\bbbc$ uniformly.

Let us first note that there exists $R_{0}\ge0$ such that $\alpha_{i}(D_{R})\subseteq(D_{R})$
for each $i=1,\dots,\ell$ and each $R\ge R_{0}$. The following domain
expansion lemma is straightforward to prove. \begin{lemma} \label{lem:Extending_T}
Let $g \in G_{R_{1}}$ for
some $R_{1}>R_{0}$. Then for each $R\ge R_{1}$ there exists an integer
$K\ge0$ such that $T^{k}g\in G_{R}$ for all $k\ge K$. \end{lemma} \begin{proof} The
proof is a straightforward consequence of the fact that the $\alpha_{i}$
contract uniformly. If $g\in G_{R}$, then the result holds with $k=0$.
Otherwise, let $k\ge1$ and consider a composition of $k$ contractions
chosen from the $\alpha_{i}$, possibly with repetition. The resulting
composition $\alpha_{i_{1}}\dots\alpha_{i_{k}}$ is an affine map
so we may write $\alpha_{i_{1}}\dots\alpha_{i_{k}}(z)=s_{i_{1}\dots i_{k}}z+t_{i_{1}\dots i_{k}}$,
for $i_{1},\dots,i_{k}\in\{1,\dots,\ell\}$. Now $t_{i_{1}\dots i_{k}}=\alpha_{i_{1}}\dots\alpha_{i_{k}}(0)\in D_{R_{0}}$.
Moreover, since the $\alpha_{i}$ contract uniformly on $\bbbc$,
the sequence $s_{i_{1}\dots i_{k}}\to0$ uniformly in $k$ as $k\to\infty$.
It follows immediately, that there must exist $k\ge1$ such that $\alpha_{i_{1}}\dots\alpha_{i_{k}}(D_{R})\subseteq D_{R_{1}}$
for all $i_{1},\dots,i_{k}\in\{1,\dots,\ell\}$. It follows that for all
$k$ large enough, we have $T^{k}g\in G_{R}$, as claimed. \end{proof}

In the application we shall consider in the next subsection, the operator
$T$ fails to be a contraction, because it is not a contraction on
constant functions. To solve this problem, for any affine CSO $T$,
we introduce a new derived operator which is a contraction, and which
shares certain fixed points with $T$. We give the construction for
general $\ell\ge2$, although in our application we shall specialise
to the binary case $\ell=2$.

For $1\le j\le\ell$, we define the operator $T_{j}$ by 
\begin{equation}
T_{j}f=Tf-\frac{1}{L}\sum_{i=1,i\not=j}^{\ell}a_{i}Tf(\alpha_{i}(z_{j}))\,,\quad L=\sum_{i=1,i\not=j}^{\ell}a_{i}\label{eq:Tj}
\end{equation}
provided $L\not=0$. We note that, if $a_{j}=1$, a fixed point of
$T_{j}$ is also a fixed point of $T$. For, suppose $T_{j}f=f$.
Then 
\begin{equation}
f(z)=Tf(z)-\frac{1}{L}\sum_{i=1,i\not=j}^{\ell}a_{i}Tf(\alpha_{i}(z_{j}))\,.\label{eq:Tjf=00003D00003D00003Df}
\end{equation}
Taking a weighted sum of this equation evaluated at $\alpha_{i}(z_{j})$,
gives 
\begin{equation}
\sum_{i=1,i\not=j}^{\ell}a_{i}f(\alpha_{i}(z_{j}))=\sum_{i=1,i\not=j}^{\ell}a_{i}\left(Tf(\alpha_{i}(z_{j}))-\frac{1}{L}\sum_{k=1,k\not=j}^{\ell}a_{k}Tf(\alpha_{k}(z_{j}))\right)=0\,.\label{eq:sum_ai_f_alpha_zi}
\end{equation}
Hence, using $\alpha_{j}(z_{j})=z_{j}$ we obtain 
\begin{equation}
\begin{split}
f(z_{j})=Tf(z_{j})-\frac{1}{L}\sum_{i=1,i\not=j}^{\ell}a_{i}Tf(\alpha_{i}(z_{j}))  \\
=a_{j}f(z_{j})+\sum_{i=1,i\not=j}^{\ell}a_{i}f(\alpha_{i}(z_{j}))-\frac{1}{L}\sum_{i=1,i\not=j}^{\ell}a_{i}Tf(\alpha_{i}(z_{j})),
\end{split}
\end{equation}
whence, since $a_{j}=1$, 
\begin{equation}
\frac{1}{L}\sum_{i=1,i\not=j}^{\ell}a_{i}Tf(\alpha_{i}(z_{j}))=\sum_{i=1,i\not=j}^{\ell}a_{i}f(\alpha_{i}(z_{j}))=0\,,
\end{equation}
from \eqref{eq:sum_ai_f_alpha_zi}. It follows that $Tf=f$, as claimed.

If $T$ is a binary CSO, then we can write the operator $T_{j}$ as
$T_{c}$ where the modified operator $T_{c}$ is given by $T_{c}f=Tf-Tf(c)$
with $c=\alpha_{i}(z_{j})$ and where now $\{i,j\}$ is precisely
$\{1,2\}$. It is immediate that, for a binary CSO, any fixed point
$f$ of $T_{c}$ is a fixed point of $T$ and $f(c)=0$ by~\eqref{eq:sum_ai_f_alpha_zi}.

\subsection{Fixed points of the operator $M$}

We now apply the theory we have developed to find fixed points of
the operator $M$ introduced in Section~\ref{sec:seeded-fixed-point-thy}.
Recall that $M=\phi_{1}^{*}+\phi_{2}^{*}$ with $\phi_{1}(z)=-\omega z,\phi_{2}(z)=\omega^{2}z+\omega$,
where $\omega=\frac{1}{2}(\sqrt{5}-1)$. For consistency with~\cite{Mestel2000,Dalton-Mestel2003},
we use the notation $\phi_{i}=\alpha_{i}$ for $i=1,2$.

Now, in the notation used above, $s=\max(\omega,\omega^{2})=\omega$.
Let $\omega<\mu<1$. It follows that, for $N$ sufficiently large,
$M(Z_{n})\le\mu^{n}R^{n}$, for $n\ge N$ and $R\ge R_{0}$. In fact
it is readily seen that $N=2$ suffices when $R\ge1.9009$. For $0\le n\le1$,
we calculate as follows. If $n=1$, $\left|\sum_{i}\lambda_{i}s_{i}^{n}\right|=\left|(-\omega)+\omega^{2}\right|=\omega^{3}<1$
and we have a contraction for $R$ sufficiently large. It is straightforward
to verify that is sufficient to take $R\ge1.619$. We therefore need
to choose $R\ge1.9009$. However, for $n=0$, $\left|\sum_{i}\lambda_{i}s_{i}^{n}\right|=2>1$
and we do not have a contraction.

Let us consider the operator $M_{c}$ introduced at the end of the
last section and given by $M_{c}f=Mf-Mf(c)$. Since $a_{1}=a_{2}=1$,
in this case, we can take in turn $j=1,2$ and choose, in turn, $c=c_{1},c_{2}$,
where $c_{1}=\phi_{1}(1)=-\omega$ and $c_{2}=\phi_{2}(0)=\omega$.
We note that, for any $c$, $M_{c}$ is itself a (degenerate) CSO.
Indeed, $M_{c}=\phi_{1}^{*}+\phi_{2}^{*}-(\phi_{1}c)^{*}-(\phi{}_{2}c)^{*}=\sum_{1}^{4}a_{i}\alpha_{i}$.
The last two terms $(i=3,4)$ are constants, so are degenerate affine
contractions.

Let us now construct fixed points of $M_{c}$ and hence of $M$. First
we note that $M_{c}Z_{0}=0$, and, from the above calculations, we
see that $M_{c}$ contracts the functions $Z_{n}:\, z\longmapsto z^{n}$
for $n\ge1$. We deduce from Corollary~\ref{cor:TcontractionforlargeR-1}
that $M_{c}$ is a contraction on $G_{R}$ for large enough $R$.
Hence $M_{c}$ has no non-trivial analytic fixed points in $G_{R}$,
for $R$ large enough.

Our first task is to consider a space of seed functions for $M$.
Since $a_{1}$ $=$ $a_{2}$ $=$ $1$, we can look for seeds with
logarithmic singularities at $0$, $1$, the fixed points of $\phi_{1}$
and $\phi_{2}$ respectively. Indeed for simple unbounded singularities,
the space of seed functions for $M_{c}$ is the span $<\log z,\log(z-1)>$.
Hence, to find the fixed points of $M_{c}$ we apply the previous
theory to obtain a fixed point for each of the two basis seed functions
$\log z$ and $\log(z-1)$.

We calculate $M$ acting on the two basis (generalised) seed functions.
We have $M\log z=$ $\log(-\omega z)+\log(\omega^{2}z+\omega)$ $=$
$\log z+\log(1+\omega z)+b_{1}$, where $b_{1}$ is a constant. Similarly
$M\log(z-1)=\log(-\omega z-1)+\log(\omega^{2}z+\omega-1))$ $=$ $\log(z-1)+\log(1+\omega z)+b_{2}$,
$b_{2}$ constant (where we have used $1-\omega=\omega^{2}$). It
follows readily that in both cases we have (using the notation of Lemma \ref{lem:Extending_T}) $\overline{M_{c}}\log z$,
$\overline{M_{c}}\log(z-1)$ $\in G_{R_1}$, provided $1=R_0<R_1<\omega^{-1}$.
Because of the restriction on $R_1$, we cannot necessarily use Theorem~\ref{thm:fixed_point_thm}
directly, but if not we can use Corollary~\ref{cor:FixedPointThmExtension}
instead. For convenience, in what follows, we write $f(z)$ to represent one 
of the generalised seed functions $\log z$ and $\log(z-1)$ and we
write $\tilde{f}(z)=\log(1+\omega z)$.

Using Corollary~\ref{cor:TcontractionforlargeR-1} we choose $R$
sufficiently large so that $M_{c}$ is a contraction on $G_{R}$. If we can choose $1<R<\omega^{-1}$, we can use 
Theorem~\ref{thm:fixed_point_thm} directly to obtain a fixed point of $M_{c}$. Otherwise
using Lemma~\ref{lem:Extending_T}, we set $k\ge0$ such that
$M_{c}^{k}\overline{M_{c}}f\in G_{R}$. We note that in fact $k\ge1$ since
$\tilde{f}\not\in G_{R}$, because $R\ge\omega^{-1}$. It also follows
that $M_{c}^{k}(f)$ is a seed function since $\overline{M_{c}}M_{c}^{k}f\in G_{R}$.
A fixed point of $M_{c}$ is now obtained from this seed function by applying
Corollary~\ref{cor:FixedPointThmExtension}.

We can obtain some explicit expansions for the fixed points. From
the above calculations, we have $Mf=f+\tilde{f}+b$, where $b$ is
a constant, whence $M_{c}f=I_{c}f+I_{c}\tilde{f}$. Hence we readily obtain for $k\ge1$
that $M_{c}^{k}f=I_{c}f+I_{c}\tilde{f}+\sum_{n=1}^{k-1}M_{c}^{n}\tilde{f}$,
and so $\overline{M_{c}}M_{c}^{k}f=-M_{c}^{k}\tilde{f}$. We conclude
that $M_{c}^{k}\tilde{f}\in G_{R}$.

Let 
\begin{equation}
f_{*}=I_{c}f+I_{c}\tilde{f}+\sum_{n=1}^{\infty}M_{c}^{n}\tilde{f}\,.\label{eq:f*}
\end{equation}
The infinite sum converges because $M_{c}$ is a contraction on $G_{R}$
and $M_{c}^{k}\tilde{f}\in G_{R}$. Now, 
\begin{equation}
M_{c}f_{*}=M_{c}f+M_{c}\tilde{f}+\sum_{n=2}^{\infty}M_{c}^{n}\tilde{f}=I_{c}f+I_{c}\tilde{f}+\sum_{n=1}^{\infty}M_{c}^{n}\tilde{f}=f_{*}\,.
\end{equation}
since $M_{c}f=I_{c}(f+\tilde{f})$ $=$ $I_{c}f+I_{c}\tilde{f}$.

As remarked above, we choose in turn $f(z)=\log z$ and $c=$ $c_{1}=\phi_{1}(1)=-\omega$,
$f(z)=\log(z-1)$, $c_{2}=\phi_{2}(0)=\omega$. From~\eqref{eq:f*}
and noting that $M_{c}^{n}g=M^{n}(g)-(M^{n}g)(c)$, this gives a fixed
point space for $M$ of $<f_{1},f_{2}>$, where 
\begin{equation}
f_{1}(z)=\log\frac{z}{-\omega}+\sum_{n=0}^{\infty}\sum_{\underline{i}\in I^{n}}\log\frac{1+\omega\phi_{\underline{i}}z}{1+\omega\phi_{\underline{i}}(-\omega)}
\end{equation}
and 
\begin{equation}
f_{2}(z)=\log\frac{z-1}{\omega-1}+\sum_{n=0}^{\infty}\sum_{\underline{i}\in I^{n}}\log\frac{1+\omega\phi_{\underline{i}}z}{1+\omega\phi_{\underline{i}}(\omega)}\,,
\end{equation}
where $\underline{i}=(i_{1},i_{2},\ldots,i_{n}),\ I^{n}=\{1,2\}^{n},\ \phi_{\underline{i}}=\circ_{j=1}^{n}\phi_{i_{j}}$
$=$ $=\phi_{i_{1}}\dots\phi_{i_{n}}$, for $n>0$, and the identity
map when $n=0$. We note that $f_{2}$ is the fixed point already
reported by Mestel \emph{et al} in~\cite{Mestel2000}.

A particularly elegant example of a fixed point of $M$ is obtained
by putting $f_{3}=f_{1}-f_{2}$ to give the fixed point $f_{3}(z)$
$=$ $\left[\log\frac{z}{z-1}.\frac{\omega-1}{-\omega}+W\right]$
where $W=\sum_{n=0}^{\infty}\sum_{\underline{i}\in I^{n}}\log\frac{1+\omega\phi_{\underline{i}}\omega}{1+\omega\phi_{\underline{i}}(-\omega)}$.
Since this is a fixed point we also have $f_{3}=Mf_{3}=\left[M\log\frac{z}{z-1}.\frac{\omega-1}{-\omega}+2W\right]$
$=$ $\log\frac{-\omega z}{-\omega z-1}.\frac{\omega^{2}z+\omega}{\omega^{2}z+\omega-1}.\left(\frac{\omega-1}{-\omega}\right)^{2}+2W$
$=$ $\left[\log\frac{z}{z-1}.\omega^{2}+2W\right]$. Hence $W=-\log\omega$
$=\log(1+\omega)$, by the properties of $\omega$, and $f_{3}(z)=\lambda\log\frac{z}{z-1}$,
where $\lambda$ is constant. This gives us the subspace of fixed
points $<\log\frac{z}{z-1}>$ and also the identity $e^{W}=$ 
\begin{equation}
\prod_{n=0}^{\infty}\prod_{\underline{i}\in I^{n}}\left(\frac{1+\omega\phi_{\underline{i}}\omega}{1+\omega\phi_{\underline{i}}(-\omega)}\right)=1+\omega\,.
\end{equation}

Clearly we can take exponentials of the fixed points $f_{1},f_{2}$
to obtain instead fixed points of the \emph{multiplicative} functional
equation $f(z)=f(-\omega z).f(\omega^{2}z+\omega)$. The singularities
are removable and can be replaced with zeroes to obtain entire functions.
The real parts of $\exp f_{1},\exp f_{2}$ are shown in Figure \ref{fig:Graph},
and it can be seen that this is consistent with the identity $\frac{\exp f_{1}}{\exp f_{2}}=\frac{z}{z-1}$.

\begin{figure}
\noindent \begin{centering}
\includegraphics[scale=0.2]{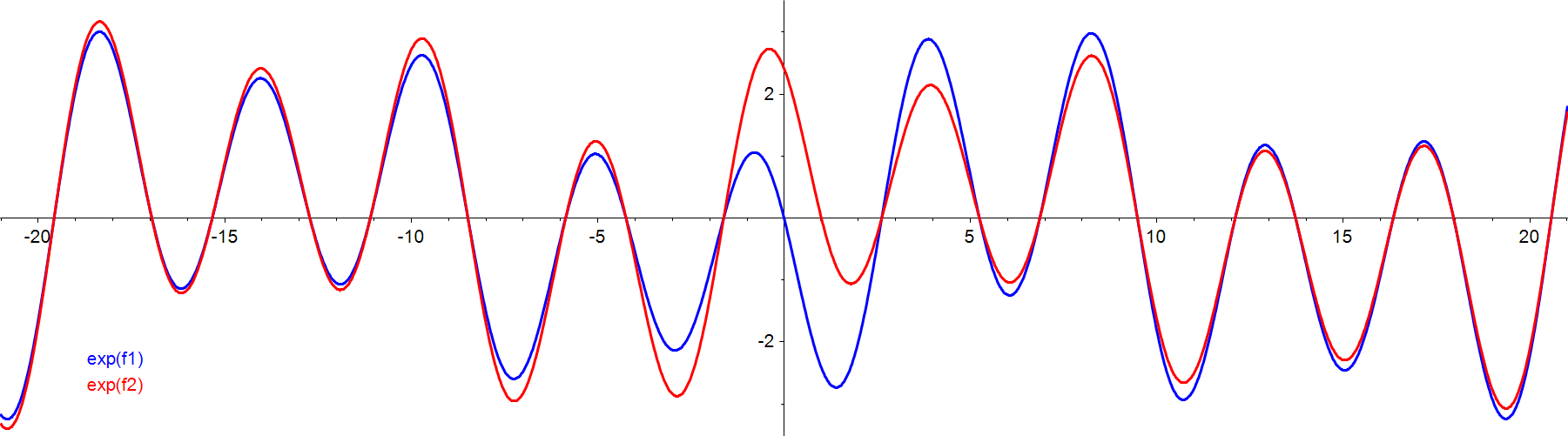} 
\par\end{centering}

\caption{\label{fig:Graph}Graph of the real parts of the two multiplicative
fixed points}
\end{figure}


\section{Conclusion and further directions for research}

In this paper we have introduced a new theory of non-zero fixed points
of linear operators $T$, showing the existence of a fixed point operator
$\widehat{T}$ derived from $T$ whose image is all the fixed points
of $T$. We have applied this theory to ``affine composition sum
operators'', a class of operators whose fixed points are important
in the study of the renormalisation of a variety of physical problems.
In particular this has enabled us, under a simple set of constraints,
to show the necessary form of fixed points which are of PESL (Pole/Essential/Simple
Log) type (see \eqref{PESL} for the full definition), these being
currently seen as the most important type. The techniques are readily
extended to functions with other types of singularity such as removable
of algebraic singularities. In addition we developed simple tests
for their existence. Finally we have applied the theory to the much
studied operator $Mf(z)=f(-\omega x)+f(\omega^{2}x+\omega)$ (see
\cite{Mestel2000,Dalton-Mestel2003} and subsequent papers) to deduce
for the first time the \emph{complete} set of \emph{simple} (see \eqref{simple})
fixed points of PESL type, including previously unknown solutions.

There are several directions for further research in this area. First,
we may extend our study to cover the full spectrum of affine CSOs.
Considered as linear operators on function spaces of analytic functions,
CSOs are compact operators and thus have discrete non-zero spectrum.
It is likely that the techniques developed in this paper may be adapted
to construct more general eigenfunctions of affine CSOs, with a view
to obtaining a full description of their spectra.

Second, it is likely that the approach of, for example,~\cite{Mestel2004a}
may be applied to understand fixed points of an affine CSO with non-simple
unbounded singularity set and, more generally, all periodic points
of an affine CSO. An full understanding of the latter is indeed necessary
for a complete description of all the fixed points of a CSO. For let
$f_{1}$, $f_{2}$ be a periodic orbit of period-2 of a CSO $T$.
Then $Tf_{1}=f_{2}$ and $Tf_{2}=f_{1}$ so that $f=f_{1}+f_{2}$
is generally a non-simple fixed point of $T$, a construction that
clearly generalises to other periods.

Third, an important future direction is to consider more general CSOs
than affine CSOs. Of course, explicit construction of fixed points
(and more general eigenfunctions) may not be in general possible for
non-affine CSOs. However a general theory may well be possible and
it may be possible to make extensive progress for special important
cases. An analogy may be drawn here with the theory of linear differential
equations. The theory of constant coefficient linear differential
equations is complete, while that for general linear equations is
less well developed except in special cases of particular interest.
Nevertheless, non-constant coefficient CSOs are of considerable interest.
For example, the full Stretch-Fold-Shear toy model studied by Gilbert~\cite{Gilbert2002,Gilbert2005}
involves a study of the spectrum of the CSO $T$ on complex-valued
$c$ functions of a real variable $x$ given by 
\begin{equation}
Tc(x)=e^{i\alpha(x-1)/2}\; c\left(\frac{x-1}{2}\right)-e^{i\alpha(1-x)/2}\; c\left(\frac{1-x}{2}\right)\,,
\end{equation}
where $\alpha\ge0$ is a real parameter, corresponding to the level
of shear in the map.

Recall that a CSO given by~\eqref{eq:CSOs} is affine if each of
the coefficients $a_{i}$ is constant and each of the maps $\alpha_{i}$
is an affine contraction. While it would certainly be interesting
to relax each of these conditions, a theory for non-constant $a_{i}$
would be of immediate application is several areas including the kinematic
dynamo theory discussed in~\cite{Gilbert2002,Gilbert2005} and in
the study of non-chaotic strange attractors~\cite{Feudel-Kuznetsov-Pikovsky2006}.

Finally, a promising area for research is to study CSOs in a wider
context than spaces of analytic functions with singularities. Indeed,
work on CSOs in spaces of piecewise constant real functions has already
found fruitful application in several fields, as detailed above. It
would be very interesting to develop a general theory of CSOs for
spaces of functions with discontinuities either of the function or
its derivatives.


\bibliographystyle{gDEA}
\bibliography{Verschueren-Mestel}


\end{document}